\documentclass[reqno,english]{amsart}
\usepackage[utf8]{inputenc}
\usepackage{amsmath,amssymb,amsthm,mathrsfs,color,times,textcomp,yfonts,mathtools,cases}

\allowdisplaybreaks[4]

\usepackage{bm}

\usepackage[T1]{fontenc} 

\usepackage{subcaption}
\usepackage[normalem]{ulem}
\usepackage[export]{adjustbox}
\usepackage{esint}
\usepackage{xcolor}
\usepackage{array}
\usepackage[colorlinks=true]{hyperref}
\hypersetup{urlcolor=blue, citecolor=blue, linkcolor=red}

\hypersetup{
colorlinks=true,
linkcolor=red
}
\usepackage{indentfirst}
\usepackage{graphicx}
\usepackage{float}
\numberwithin{equation}{section}
\newtheorem{theorem}{Theorem}[section]
\newtheorem{lemma}[theorem]{Lemma}

\newtheorem{proposition}[theorem]{Proposition}

\newtheorem{conjecture}{Conjecture}

\newtheorem{othertheorem}{Theorem}

\newtheorem{remark}{Remark}[section]

\newcommand{\Rom}[1]{%
\textup{\uppercase\expandafter{\romannumeral#1}}%
}
\newcommand\norm[1]{\left\lVert#1\right\rVert}

\newcommand*{\ud}{d}

\usepackage{mathtools}
\usepackage{thmtools}
\declaretheoremstyle[headfont=\normalfont]{normalhead}
\usepackage{color}
\usepackage{morefloats}
\usepackage{textcomp}
\usepackage{url}

\title[Sharp Beckner's Inequalities]{Sharp Beckner's Inequalities for Axially Symmetric Functions on $\mathbb{S}^N$}

\author{Changfeng Gui}
\address{Department of Mathematics, University of Macau, Taipa, Macau}
\email{changfenggui@um.edu.mo}

\author{Tuoxin Li}
\address{Department of Mathematics \\  The Chinese University of Hong Kong \\ Shatin \\ NT \\ Hong Kong}
\email{txli@math.cuhk.edu.hk}

\author{Juncheng Wei}
\address{Department of Mathematics \\  The Chinese University of Hong Kong \\ Shatin \\ NT \\ Hong Kong}
\email{wei@math.cuhk.edu.hk}

\author{Zikai Ye}
\address{Department of Mathematics \\  The Chinese University of Hong Kong \\ Shatin \\ NT \\ Hong Kong}
\email{zkye@math.cuhk.edu.hk}
\begin{document}



\begin{abstract}
We prove that for every integer $N\geq 3$ and $\alpha\geq \frac{1}{2}$, Beckner's inequality
\begin{equation*}
\frac{\alpha}{2}\int_{\mathbb{S}^N}u(P_{N}u)
dw+(N-1)!\int_{\mathbb{S}^N}u
dw-\frac{(N-1)!}{N}\ln\int_{\mathbb{S}^N}e^{Nu}
dw\geq 0
\end{equation*}
holds for any axially symmetric $u\in H^{\frac{N}{2}}(\mathbb{S}^N)$ whose center of mass is at the origin. The proof is mainly based on a weighted $\ell ^2$ estimate on Gegenbauer coefficients and a rigidity theorem for stable critical points. Hence, we answer the generalized Chang-Yang conjecture positively in the axially symmetric case for every integer $N\geq 3$. 
\end{abstract}

\maketitle
\hypersetup{linkcolor=black}
\tableofcontents

\section{Introduction and main results}

\subsection{Sharp Beckner's inequality}

For $\alpha>0$, we consider the functional
\begin{equation*}
J_{\alpha,N}(u):=\frac{\alpha}{2}\int_{\mathbb{S}^N}u(P_{N}u)
dw+(N-1)!\int_{\mathbb{S}^N}u
dw-\frac{(N-1)!}{N}\ln\int_{\mathbb{S}^N}e^{Nu}
dw
\end{equation*}
for $u\in H^{\frac{N}{2}}(\mathbb{S}^N)$, where $dw$ denotes the normalized surface measure on $\mathbb{S}^N$ so that $\int_{\mathbb{S}^N}dw=1$. Here
\begin{equation*}
\begin{aligned}
    P_N:=
    \begin{cases}
        \prod_{k=0}^{\frac{N-2}{2}}(-\Delta+k(N-k-1)),&\text{ for }N\text{ even}, \\
        \left(-\Delta+(\frac{N-1}{2})^2\right)^{\frac{1}{2}}\prod_{k=0}^{\frac{N-3}{2}}(-\Delta+k(N-k-1)),&\text{ for }N\text{ odd}
    \end{cases}
\end{aligned}
\end{equation*}
represents the Paneitz operator on $\mathbb{S}^N$ introduced by Paneitz in \cite{Paneitz2008}.

In his seminal paper \cite{Beckner1993}, Beckner proved the higher-order Moser-Trudinger-type inequality
\begin{equation*}
J_{1,N}(u)\geq 0,    \qquad u\in H^{\frac{N}{2}}(\mathbb{S}^N).
\end{equation*}
This inequality is referred to as Beckner's inequality. 

The Paneitz operator $P_N$ belongs to the family of conformally invariant operators $P_{N,g,k}$ constructed by Graham, Jenne, Mason, and Sparling on a Riemannian manifold $(M^N, g)$ in \cite{GJMS1992}, called the GJMS operator. When $k=\frac{N}{2}$, under conformal change $\tilde g=e^{2\omega}g$, its covariance law is
\begin{equation*}
P_{N,\tilde g,}(\phi)=e^{-N\omega}P_{N,g}(\phi).
\end{equation*}
We refer to \cite{CLY2019, CM2023, ChangYang1995, ChangYang1997, DHL2000,DM2008, FG2013, GHX2021, GurMal2015, LiXiong2019, Mal2006, WX1998} and the references therein for results and background on $Q$-curvature problems and GJMS operators. 

In addition, let $\xi=(\xi_1,\dots,\xi_{N+1})\in \mathbb{S}^N$,
A. Chang and P. Yang \cite{ChangYang1995} proved that, if $u$ belongs to the zero center-of-mass set
\begin{equation*}
\mathcal{L}_{N}=\left\{u\in H^{\frac{N}{2}}(\mathbb{S}^N)\ :\ \int_{\mathbb{S}^N}e^{Nu} \xi_j dw=0,\ j=1,\cdots, N+1 \right \},
\end{equation*}
then for any $\alpha\geq \frac{1}{2}$, there exists a constant $C(\alpha,N)\geq0$ such that 
\begin{equation*}
J_{\alpha,N}(u)\geq -C(\alpha,N)    
\end{equation*}
for any $u\in \mathcal{L}_{N}$. This lower bound motivates the generalized Chang-Yang conjecture that, as in dimension two, $C(\alpha,N)$ can be chosen to be zero.

\begin{conjecture}[Generalized Chang-Yang conjecture]\label{Chang-Yang conj}
For $\alpha \geq \frac{1}{2}$, 
\begin{equation*}
    \inf_{u \in \mathcal{L}_{N}} J_{\alpha,N}(u)=0.
\end{equation*}
\end{conjecture}

A critical point of $J_{\alpha,N}$ under the zero center of mass constraint satisfies the following $Q$-curvature-type equation on $\mathbb{S}^N$:
\begin{equation*}
\alpha P_N u+(N-1)!(1-\frac{e^{Nu}}{\int_{\mathbb{S}^N}e^{Nu}dw})=\sum_{i=1}^{N+1}a_i x_i e^{Nu} \ \mbox{on} \ \mathbb{S}^N
\end{equation*}
for some constants $a_i$, $i=1,\dots, N+1$.

Due to conformal invariance of $J_{1,N}$, A. Chang and P. Yang \cite{ChangYang1995} showed that the following Kazdan-Warner condition 
\begin{equation*}
\int_{\mathbb S^N}\langle \nabla Q, \nabla x_i\rangle e^{Nu} \ud w=0,\ i=1,\dots, N+1,
\end{equation*}
holds for the prescribed $Q$-curvature equation
\begin{equation*}
P_{N}u + (N-1)!- Qe^{Nu}=0 \text{ on }\mathbb S^N.
\end{equation*}
It follows that for $\frac{1}{2}\leq \alpha\leq 1$, all Lagrange multipliers $a_i$ vanish. See also Wei-Xu \cite{WX1998} for more details. Then the Euler-Lagrange equation becomes
\begin{equation}\label{paneitz}
\alpha P_N u+(N-1)!(1-\frac{e^{Nu}}{\int_{\mathbb{S}^N}e^{Nu}dw})=0 \ \mbox{on} \ \mathbb{S}^N.
\end{equation}

For $\frac{1}{2}\leq \alpha<1$, if \eqref{paneitz} admits only constant solutions, then Conjecture \ref{Chang-Yang conj} is valid. When $\alpha<1$ is sufficiently close to $1$, Wei and Xu \cite{WX1998} proved that all solutions to \eqref{paneitz} are constants, but the full range $\alpha\in[\frac{1}{2},1)$ remains open.

On $\mathbb{S}^2$, the original conjecture of A. Chang and P. Yang in \cite{ChangYang1987, ChangYang1988} is that
\begin{conjecture}[Chang-Yang conjecture]
For $\alpha \geq \frac{1}{2}$, 
\begin{equation*}
    \inf_{u \in \mathcal{L}_{2}} J_{\alpha,2}(u)=0.
\end{equation*}
\end{conjecture}

The corresponding problem on $\mathbb{S}^2$ is known as the Nirenberg problem:
\begin{equation*}
-\alpha \Delta u + 1-  \frac{e^{2u}}{\int_{\mathbb{S}^2} e^{2u}}=0 \ \ \mbox{on} \ \mathbb{S}^2.
\end{equation*}
This problem has been extensively studied over the past four decades. See \cite{ChangYang1987, ChangYang1988, JinLiXiong2017} and the references therein. Feldman, Froese, Ghoussoub, and the first author  \cite{FFGG1998} first established the conjecture for axially symmetric functions when $\alpha >\frac{16}{25}-\epsilon$. Here we say a function $u$ on $\mathbb{S}^2$ is axially symmetric if, up to a rotation, $u=u(x)$ for $x=x_{1}\in (-1,1)$. The first and the third authors \cite{GW2000} subsequently proved the sharp axially symmetric version of the conjecture.  Later, Ghoussoub and Lin \cite{GL2010} showed that all critical points of $J_{\alpha,2}$ are axially symmetric and hence the conjecture holds true for $\frac{2}{3}-\epsilon<\alpha<1$. Finally, the first author and Moradifam \cite{GM2018} established the sphere covering inequality to prove that all solutions are axially symmetric. With the help of the first and the third authors' result \cite{GW2000} for axially symmetric case, they showed the full conjecture.
Furthermore, Shi, Sun, Tian and Wei \cite{SSTW2019} showed that all even solutions are axially symmetric for $\frac{1}{4}\leq \alpha<1$. For more general results on improved Moser-Trudinger-Onofri inequalities on $\mathbb{S}^2$ and their connections with the Szeg\"o limit theorem, see \cite{ChangGui202, ChangHang2022}.

For the higher-dimensional equation \eqref{paneitz} on $\mathbb{S}^N$, as mentioned above, the third author and Xu proved the conjecture when $\alpha$ is close to $1$. Since direct treatment of Conjecture \ref{Chang-Yang conj} seems difficult, in view of the work of Ghoussoub-Lin \cite{GL2010} and  the work of the first author and Moradifam \cite{GM2018}, the conjecture under axial symmetry is a natural and crucial first step in higher dimensions. Several results have been obtained for axially symmetric solutions with $N=4,6,8$. Gui-Hu-Xie \cite{GHW2022} obtained non-constant solutions by bifurcation methods for $\frac{1}{N+1}<\alpha<\frac{1}{2}$. They also proved the axially symmetric version of the conjecture for $\alpha\geq 0.517$ ($N=4$), $\alpha \geq 0.6168$ ($N=6$) and $ \alpha \geq 0.8261$ ($N=8$). The sharp bound $\alpha\geq \frac{1}{2}$ is obtained by Li-Wei-Ye \cite{LWY2022} ($N=4$) and Gui-Li-Wei-Ye \cite{GLWY2025} using refined estimates on Gegenbauer polynomials. 

For odd $N$, there are few results on the sharp Beckner's inequality.  When $N=1$, Beckner's inequality becomes the classical Lebedev-Milin inequality.

\begin{othertheorem}[Lebedev-Milin inequality]
    For any $u\in H^1(\mathbb D)$ with $\int_{\mathbb S^1}u\ud \theta=0$, where $\mathbb D$ is the unit disk in $\mathbb R^2$,
    \begin{equation*}
        \log\left(\frac{1}{2\pi}\int_{\mathbb S^1}e^u \ud \theta\right)\leq \frac{1}{4\pi}\norm{\nabla u}_{L^2(\mathbb D)}^{2}.
    \end{equation*}
\end{othertheorem}

Using Szeg\"o limit theorem, Widom \cite{Widom1988} improved the best constant under extra orthogonality conditions.
\begin{othertheorem}\label{szego}
    For any $u\in H^1(\mathbb D)$ with $\int_{\mathbb S^1}u\ud \theta=0$ and $\int_{\mathbb S^1}e^{u}e^{ik\theta}\ud \theta=0$, for $k=1,\dots,m$,
    \begin{equation*}
        \log\left(\frac{1}{2\pi}\int_{\mathbb S^1}e^u \ud \theta\right)\leq \frac{1}{4(m+1)\pi}\norm{\nabla u}_{L^2(\mathbb{D})}^{2}.
    \end{equation*}
    Equivalently, for $ \alpha \geq \frac{1}{m+1}$,
    \begin{equation*}
        \frac{\alpha }{2}\int_{\mathbb S^1}(P_1 u)u\ud w+\int_{\mathbb S^1}u\ud w-\ln\int_{\mathbb S^1}e^{u}\ud w \geq 0.
    \end{equation*}
\end{othertheorem}

For odd $N\geq 3$, as mentioned before, Wei-Xu \cite{WX1998} proved the generalized Chang-Yang conjecture when $\alpha$ is close to $1$. Zhang \cite{Zhang2025} recently proved the rigidity for $\alpha>1$. To the best of our knowledge, there are no other results. One of the difficulties is the non-locality of the Paneitz operator $P_N$ when $N$ is odd.

In this paper, we focus on axially symmetric solutions on $\mathbb{S}^N$. For such functions $u$ of $x=\xi_1\in[-1,1]$,  the Paneitz operator $P_N$ can be written as
\begin{equation*}
P_{r,N}u=   (-1)^{\frac{N}{2}}[(1-x^2)^{\frac{N}{2}}u']^{(N-1)} 
\end{equation*}
when $N$ is even and 
\begin{equation*}
\begin{aligned}
P_{r,N}u
=&
\left(-(1-x^2)\frac{d^2}{dx^2}+Nx\frac{d}{dx}+(\frac{N-1}{2})^2\right)^{\frac{1}{2}}\\
&\cdot \prod_{k=0}^{\frac{N-3}{2}}\left(-(1-x^2)\frac{d^2}{dx^2}+Nx\frac{d}{dx}+k(N-k-1)\right)u    
\end{aligned}
\end{equation*}
when $N$ is odd.

The Beckner functional $J_{\alpha,N}$ reduces to
\begin{equation*}
\begin{aligned}
\mathcal{I}_{\alpha,N}(u)
&=\frac{\alpha}{2}\int_{-1}^{1}(1-x^2)^{\frac{N-2}{2}}uP_{r,N}u \ d x+(N-1)!\int_{-1}^{1}(1-x^2)^{\frac{N-2}{2}}u\ d x\nonumber\\
&-\frac{(N-1)!\sqrt{\pi}\Gamma(\frac{N}{2})}{N\Gamma(\frac{N+1}{2})}\log\left(\frac{\Gamma(\frac{N+1}{2})}{\sqrt{\pi}\Gamma(\frac{N}{2})}\int_{-1}^{1}(1-x^2)^\frac{N-2}{2} e^{Nu} d x\right)    
\end{aligned}
\end{equation*}
and
the zero-center-of-mass restriction set becomes
\begin{equation}\label{AxialLrN}
\mathcal{L}_{r,N}=\left\{u\in H^{\frac{N}{2}}(\mathbb{S}^N):\ u=u(x)\text{ and }\int_{-1}^{1}x(1-x^2)^{\frac{N-2}{2}}e^{Nu} d x=0\right\}.
\end{equation}

We close the remaining gap for general $N\geq 3$ and hence prove the generalized Chang-Yang conjecture in the axially symmetric case for every integer $N\geq 3$. 
\begin{theorem}\label{main}
Let $N\geq 3$ be an integer and $\alpha\geq \frac{1}{2}$. Then
\begin{equation*}
\inf_{ u\in {\mathcal{L}_{r,N}}} \mathcal{I}_{\alpha,N} (u)=0.    
\end{equation*}
Moreover, for $u\in \mathcal{L}_{r,N}$, $\mathcal{I}_{\alpha,N}(u)=0$ if and only if $u$ is a constant.
\end{theorem}
\begin{remark}
In view of \cite[Proposition 1.5]{GLWY2025}, the constant $\frac{1}{2}$ is sharp when $N$ is even. Proposition \ref{optimalodd} below shows $\frac{1}{2}$ is also sharp when $N$ is odd.
\end{remark}
\begin{proposition}\label{optimalodd}
Let $N\geq 3$ be an odd integer. If $\mathcal{I}_{\alpha,N}(u)\ge 0$ for all $u\in \mathcal{L}_{r,N}$, then $\alpha\ge \frac{1}{2}$.
\end{proposition}
We postpone its proof to Appendix \ref{optimal}.

Furthermore, a critical point $u$ of $\mathcal I_{\alpha,N}$ restricted to the set $\mathcal{L}_{r,N}$ solves
\begin{equation}\label{axial}
\alpha P_{r,N}u+(N-1)!-\frac{(N-1)!}{\gamma}e^{Nu}=0,\ x\in(-1,1),
\end{equation}
where
\begin{equation*}
\gamma=\int_{\mathbb{S}^N}e^{Nu} dw.
\end{equation*}

To prove Theorem \ref{main} for $N=4,6,8$, previous papers \cite{GHX2021, GHW2022, GLWY2025, LWY2022} showed \eqref{axial} only admits constant solutions within $\mathcal{L}_{r,N}$ in corresponding dimensions. They used the lower bound of an auxiliary quantity $D_N$ to generate a series of inequalities and proved $a=\frac{N}{N+1}(1-\alpha\beta)\leq \frac{d_0}{\lambda_n}$ for some $d_0$ and $\lambda_n=O(n^2)$ by induction. Here $\beta$ is the first Gegenbauer coefficient of $G=(1-x^2)u'$. Instead, with a weighted $\ell^{2}$ estimate of the Gegenbauer coefficients of $G$, we are able to derive a lower bound of $a$.

On the other hand, the previous works \cite{GHX2021, GHW2022, GLWY2025, LWY2022} adopt an integration-by-parts method to estimate the cubic term of $G=(1-x^2)u'$
\begin{equation*}
I_N=(-1)^{\frac{N}{2}}\int_{-1}^{1}(1-x^2)^{\frac{N-2}{2}}G^2[(1-x^2)^{\frac{N-2}{2}}G]^{(N-1)} dx
\end{equation*}
to prove every solution to \eqref{axial} is a constant. However, it becomes increasingly involved as the dimension grows. 

To avoid such dimension-dependent complexity to prove every critical point of $\mathcal{I}_{\alpha,N}$ is constant, inspired by the work of Frank and Lieb \cite{FrankLieb2012},  we find that it is enough to show that every non-constant critical point is unstable. This is one of the key ingredients. The relevant statement is as follows.
\begin{theorem}\label{instability}
Let $N\geq 3$ and $\frac{1}{2}\leq \alpha\leq 1$. If a smooth critical point of the functional $\mathcal I_{\alpha, N}$ restricted to $\mathcal{L}_{r,N}$ is a non-constant function, then it is unstable. Consequently, every stable critical point restricted to $\mathcal{L}_{r,N}$ is constant.
\end{theorem}

\subsection{Outline of the paper}

This paper is organized as follows. In Section \ref{preliminary}, we gather some properties of Gegenbauer polynomials, expand $G$ in terms of Gegenbauer polynomials, and establish some useful integral identities. We also establish the existence and regularity of a minimizer. A weighted $\ell^2$ estimate of Gegenbauer coefficients of $G$ is established in Section \ref{weightedl^2est}. Section \ref{Sec: instability} proves the instability of non-constant functions and proves Theorem \ref{instability}. Proof of Theorem \ref{main} is presented in Section \ref{proofmain}. The optimality of $\alpha=\frac{1}{2}$ is shown in Appendix \ref{optimal}. Appendix \ref{Technicalities} includes some auxiliary technicalities.

\section{Preliminaries}\label{preliminary}
In this section, we collect some properties of Gegenbauer polynomials and some known facts about the equation.

\subsection{Gegenbauer polynomials}

The Gegenbauer polynomial of order $\nu$ and degree $k$ (see \cite{Mori1998}) is given by
\begin{equation}\label{Rodrigues}
C_{k}^{\nu}(x)=\frac{(-1)^k}{2^k k!}\frac{\Gamma(\nu+\frac{1}{2})\Gamma(k+2\nu)}{\Gamma(2\nu)\Gamma(\nu+k+\frac{1}{2})}(1-x^2)^{-\nu+\frac{1}{2}}\frac{d^k}{dx^k} (1-x^2)^{k+\nu-\frac{1}{2}}.
\end{equation}

The derivative of $C_{k}^{\nu}$ satisfies
\begin{equation}\label{201}
\frac{d}{dx}C_{k}^{\nu}(x)=2\nu C_{k-1}^{\nu+1}(x).
\end{equation}

Let $F_k^\nu$ be the normalization of $C_{k}^{\nu}$ such that $F_k^\nu(1)=1$, i.e.
\begin{equation*}
F_k^\nu=\frac{k!\Gamma(2\nu)}{\Gamma(k+2\nu)} C_{k}^{\nu}.
\end{equation*}
Then $F_k^\nu$ satisfies the following differential equation
\begin{equation*}
(1-x^2)(F_{k}^\nu)''-(2\nu+1)x(F_k^\nu)'+k(k+2\nu)F_k^\nu=0,
\end{equation*}
and \eqref{201} becomes
\begin{align*}
(F_{k}^\nu)'=\frac{k(k+2\nu)}{2\nu+1}F_{k-1}^{\nu+1}.
\end{align*}


On $\mathbb{S}^N$, the corresponding Gegenbauer polynomial is $C_{k}^{\frac{N-1}{2}}$. To simplify notation, in what follows we will write $F_k$ for $F_k^\frac{N-1}{2}$, and there should be no danger of confusion.

The function $F_k$ satisfies the orthogonality condition
\begin{equation}\label{Fknorm}
\begin{aligned}
    \int_{-1}^{1}(1-x^2)^{\frac{N-2}{2}}F_{k}F_ldx&=\frac{2^{N-1}\Gamma(\frac{N}{2})^2 \Gamma(k+1)}{(k+N-2)!(2k+N-1)}\delta_{kl},
\end{aligned}    
\end{equation}
where $\lambda_k=k(k+N-1)$. Furthermore, it is an eigenfunction of the Paneitz operator:
\begin{equation}\label{Paneitzeigen}
P_{r,N} F_k=
\begin{cases}
\frac{\Gamma(k+N)}{\Gamma(k)}F_k,&\qquad k\geq 1,\\
0,&\qquad k=0.
\end{cases}
\end{equation}

Another useful function is $F_k'$. For simplicity, in the rest of the paper, we normalize $\Tilde F_k'$ as 
\begin{equation*}
\Tilde{F}_{k}'=\frac{N}{\lambda_k}F_k'=\frac{k   !\Gamma(N+1)}{\Gamma(k+N-1)\lambda_k}C_{k-1}^{\frac{N+1}{2}}
\end{equation*}
so that $\Tilde{F}_{k}'(1)=1$.

The Gegenbauer polynomials $F_k$ and $\tilde F_k'$ also satisfy the following recurrence formulas
\begin{align}
xF_k=&\frac{k}{2k+N-1}F_{k-1}+\frac{k+N-1}{2k+N-1}F_{k+1}, \label{recurence1}\\
(1-x^2)F_k'=&\frac{k(k+N-1)}{2k+N-1}(F_{k-1}-F_{k+1}) \label{recurence2}
\end{align}


The following Gegenbauer product formula will be useful in several kernel estimates; see \cite[Eq.~(18.17.5)]{Olver2010}.
\begin{lemma}\label{prodform}
For any $k\geq 1$, we have
\begin{equation*}
    \tilde F_k'(x)\tilde F_k'(y)=\frac{\Gamma(\frac{N+2}{2})}{\sqrt{\pi}\Gamma(\frac{N+1}{2})}\int_{-1}^{1}\tilde F_k'\left(xy+t\sqrt{(1-x^2)(1-y^2)}\right)(1-t^2)^{\frac{N-1}{2}}dt.
\end{equation*}
\end{lemma}

The product formula above yields the following result about an important kernel, which will be useful in Section \ref{weightedl^2est}.
\begin{lemma}
Let $N\geq 3$. For $-1<x,y<1$, the series
\begin{equation*}
    K_N(x,y):=\sum_{k=1}^{\infty}(2k+N-1) \tilde F_k'(x)\tilde F_k'(y)  
\end{equation*}
converges absolutely and satisfies
\begin{equation}\label{KNform}
K_N(x,y)=\frac{N\Gamma(\frac{N+2}{2})}{\sqrt{\pi}\Gamma(\frac{N+1}{2})}\int_{-1}^{1}\frac{(1-t^2)^{\frac{N-1}{2}}}{1-xy-t\sqrt{(1-x^2)(1-y^2)}}dt.
\end{equation}
Consequently, we have
\begin{equation}\label{KNxx}
    K_N(x,x)=\frac{N^2}{(N-1)(1-x^2)}
\end{equation}
for any $-1<x<1$, and
\begin{equation}\label{KNxy}
    (1-xy)K_N(x,y)\leq \frac{N^2}{N-1}
\end{equation}
for any $-1<x,y<1$. The latter inequality is strict for $x\neq y$.
\end{lemma}
\begin{proof}
We first claim that for any $x\in (-1,1)$, 
\begin{equation}\label{1-xexpand}
    f(x):=\frac{N}{1-x}=\sum_{k=1}^{\infty}(2k+N-1)\tilde F_k'(x).
\end{equation}

We first show that it holds in $L^2((-1,1),(1-x^2)^{\frac{N}{2}}dx)$. Indeed, since $N\geq 3$, we have
\begin{equation*}
f\in L^2((-1,1),(1-x^2)^{\frac{N}{2}}dx).
\end{equation*}

By \eqref{Rodrigues}, we see that
\begin{equation*}
    \tilde F_k'(x)=\frac{(-1)^{k-1}\Gamma(\frac{N}{2}+1)}{2^{k-1}\Gamma(k+\frac{N}{2})}(1-x^2)^{-\frac{N}{2}}\frac{d^{k-1}}{dx^{k-1}} (1-x^2)^{k+\frac{N}{2}-1}.
\end{equation*}
Integrating by parts $k-1$ times, we see that
\begin{equation*}
\begin{aligned}
\int_{-1}^{1}(1-x^2)^{\frac{N}{2}}f(x)\tilde F_k'(x)dx
&=\frac{N(k-1)!\Gamma(\frac{N}{2}+1)}{2^{k-1}\Gamma(k+\frac{N}{2})}\int_{-1}^{1}(1-x^2)^{\frac{N}{2}-1}(1+x)^{k+\frac{N}{2}-1}dx\\
&=\frac{N^2(k-1)!2^{N-1}\Gamma(\frac{N}{2})^2}{\Gamma(k+N)}.    
\end{aligned}
\end{equation*}

Since
\begin{equation*}
    \int_{-1}^{1}(1-x^2)^{\frac{N}{2}}\tilde F_k'(x)^2dx=\frac{2^{N+1}\Gamma(\frac{N}{2}+1)^2(k-1)!}{(2k+N-1)\Gamma(k+N)},
\end{equation*}
we have \eqref{1-xexpand} holds in $L^2((-1,1),(1-x^2)^{\frac{N}{2}}dx)$ by completeness of the Gegenbauer polynomials $\{\tilde F_k\}_{k=1}^{\infty}$ in $L^2((-1,1),(1-x^2)^{\frac{N}{2}}dx)$.

We next justify \eqref{1-xexpand} pointwise. Fix $x,y\in (-1,1)$ and set
$M=\sqrt{(1-x^2)(1-y^2)}$, $p=xy-M$ and $q=xy+M$. Then we have $-1\leq p<q\leq 1$.

Write $z=xy+Mt$, for any measurable function $H$ such that either side below is absolutely integrable, we have
\begin{equation*}
\mathcal{L}_{x,y}H:=\int_{-1}^{1}H(xy+Mt)(1-t^2)^{\frac{N-1}{2}}dt=\int_{p}^{q}H(z)\rho_{x,y}(z)dz,
\end{equation*}
where
\begin{equation*}
\rho_{x,y}(z)=\frac{1}{M}\left(1-\frac{(z-xy)^2}{M^2}\right)_+^{\frac{N-1}{2}}.
\end{equation*}
Here $r_+:=\max\{r,0\}$. For $p<z<q$, since
\begin{equation*}
1-\frac{(z-xy)^2}{M^2}=\frac{(z-p)(q-z)}{M^2},
\end{equation*}
we see that
\begin{equation*}
\rho_{x,y}(z)=M^{-N}(z-p)^{\frac{N-1}{2}}(q-z)^{\frac{N-1}{2}}.
\end{equation*}

Then we have
\begin{equation*}
(1-z^2)^{-\frac{N}{2}}\rho_{x,y}(z)^2=M^{-2N}\frac{(z-p)^{N-1}(q-z)^{N-1}}{(1+z)^{\frac{N}{2}}(1-z)^{\frac{N}{2}}}\leq M^{-2N}(z-p)^{\frac{N}{2}-1}(q-z)^{\frac{N}{2}-1}.  
\end{equation*}

Consequently, we have
\begin{equation*}
\int_{-1}^{1}(1-z^2)^{-\frac{N}{2}}\rho_{x,y}(z)^2dz\leq M^{-2N}\int_p^q(z-p)^{\frac{N}{2}-1}(q-z)^{\frac{N}{2}-1}dz=M^{-2N}(q-p)^{N-1}B(\frac{N}{2},\frac{N}{2})<\infty.
\end{equation*}
Here $B$ is the Beta function.

Then by Cauchy-Schwarz inequality, we have
\begin{equation*}
\left|\int_{p}^{q}H(z)\rho_{x,y}(z)dz \right|\leq \left(\int_{-1}^{1}(1-z)^{\frac{N}{2}}H(z)^2dz\right)^{\frac{1}{2}}\left(\int_{-1}^{1}(1-z^2)^{-\frac{N}{2}}\rho_{x,y}(z)^2dz\right)^{\frac{1}{2}}.
\end{equation*}
That is, $\mathcal{L}_{x,y}H$ is a bounded linear functional on $L^2((-1,1),(1-z^2)^{\frac{N}{2}}dz)$.

Denote
\begin{equation*}
f_m(z):=\sum_{k=1}^{m}(2k+N-1)\tilde F_k'(z).
\end{equation*}
Since $f_m\to f$ in $L^2((-1,1),(1-z^2)^{\frac{N}{2}}dz)$,  applying the inequality above to $f_m-f$, 
we have $\mathcal{L}_{x,y}f_m\to \mathcal{L}_{x,y}f$ as $m\to \infty$.

Since $f_m$ consists of finite summations, we can apply Lemma \ref{prodform} term by term to get
\begin{equation*}
\sum_{k=1}^{m}(2k+N-1) \tilde F_k'(x)\tilde F_k'(y)=\frac{\Gamma(\frac{N+2}{2})}{\sqrt{\pi}\Gamma(\frac{N+1}{2})}\mathcal{L}_{x,y}f_m.
\end{equation*}

Taking $m\to \infty$, we obtain
\begin{equation*}
\sum_{k=1}^{\infty}(2k+N-1) \tilde F_k'(x)\tilde F_k'(y)=\frac{\Gamma(\frac{N+2}{2})}{\sqrt{\pi}\Gamma(\frac{N+1}{2})}\mathcal{L}_{x,y}f=\frac{N\Gamma(\frac{N+2}{2})}{\sqrt{\pi}\Gamma(\frac{N+1}{2})}\int_{-1}^{1}\frac{(1-t^2)^{\frac{N-1}{2}}}{1-xy-Mt}dt,
\end{equation*}
which is \eqref{KNform}.

In particular, setting $x=y$ gives
\begin{equation*}
K_N(x,x)=\sum_{k=1}^{\infty}(2k+N-1) \tilde F_k'(x)^2=\frac{N\Gamma(\frac{N+2}{2})}{\sqrt{\pi}\Gamma(\frac{N+1}{2})}\int_{-1}^{1}\frac{(1-t^2)^{\frac{N-1}{2}}}{(1-x^2)(1-t)}dt=\frac{N^2}{(N-1)(1-x^2)}.
\end{equation*}
Since the partial sums on the left-hand side are nondecreasing, it holds pointwise.

By Cauchy-Schwarz inequality, we have
\begin{equation}\label{absoluteconv}
\sum_{k=1}^{\infty}(2k+N-1) |\tilde F_k'(x)\tilde F_k'(y)|\leq \sqrt{K_N(x,x)K_N(y,y)}=\frac{N^2}{(N-1)M}.
\end{equation}
Hence, $K_N(x,y)$ converges absolutely.

To see \eqref{KNxy}, since $(1-t^2)^{\frac{N-1}{2}}$ is even and $(1-x^2)(1-y^2)<(1-xy)^2$ for $x\neq y$, we have
\begin{equation*}
    \begin{aligned}
        (1-xy)K_N(x,y)
        &=\frac{N\Gamma(\frac{N+2}{2})}{\sqrt{\pi}\Gamma(\frac{N+1}{2})}\int_{-1}^{1}\frac{(1-t^2)^{\frac{N-1}{2}}}{1-\frac{M}{1-xy}t}dt\\
        &=  \frac{N\Gamma(\frac{N+2}{2})}{2\sqrt{\pi}\Gamma(\frac{N+1}{2})}\int_{-1}^{1}\left(\frac{1}{1-\frac{M}{1-xy}t}+\frac{1}{1+\frac{M}{1-xy}t}\right)(1-t^2)^{\frac{N-1}{2}}dt\\
        &=  \frac{N\Gamma(\frac{N+2}{2})}{\sqrt{\pi}\Gamma(\frac{N+1}{2})}\int_{-1}^{1}\frac{1}{1-\frac{M^2}{(1-xy)^2}t^2}(1-t^2)^{\frac{N-1}{2}}dt\\
        &\leq   \frac{N\Gamma(\frac{N+2}{2})}{\sqrt{\pi}\Gamma(\frac{N+1}{2})}\int_{-1}^{1}\frac{1}{1-t^2}(1-t^2)^{\frac{N-1}{2}}dt=\frac{N^2}{N-1}   
    \end{aligned}
\end{equation*}
and the inequality is strict when $x\neq y$.
\end{proof}

\subsection{Existence and regularity of a minimizer}
In this subsection, we prove that the infimum of $J_{\alpha,N}$ over $\mathcal{L}_{r,N}$ is attainable for $\frac{1}{2}<\alpha<1$.

\begin{proposition}\label{minimizerexist}
Let $N\geq 3$ and $\frac{1}{2}<\alpha<1$. Then the infimum of $J_{\alpha,N}$ over $\mathcal{L}_{r,N}$ is attained by an axially symmetric function $u\in H^{\frac{N}{2}}(\mathbb{S}^N)$.
\end{proposition}
\begin{proof}
Since $J_{\alpha,N}$ and $\mathcal{L}_{r,N}$ are invariant under addition of constants, we may assume
\begin{equation*}
\int_{\mathbb{S}^N}udw=0.
\end{equation*}

Choose $\alpha_0$ with $\frac{1}{2}<\alpha_0<\alpha$. By Chang-Yang \cite{ChangYang1995}, $J_{\alpha_0,N}(u)\geq -C$ for some $C=C(\alpha_0,N)$. Then we have
\begin{equation}\label{Jsplit}
J_{\alpha,N}(u)=J_{\alpha_0,N}(u)
+\frac{\alpha-\alpha_0}{2}\int_{\mathbb{S}^N}uP_{r,N}u dw
\geq-C+\frac{\alpha-\alpha_0}{2}\int_{\mathbb{S}^N}uP_{r,N}u dw.
\end{equation}

Since $J_{\alpha,N}(1)=0$ and $1\in \mathcal{L}_{r,N}$, we can choose a minimizing sequence $u_j$ with 
\begin{equation*}
\int_{\mathbb{S}^N}u_jdw=0.
\end{equation*}
and $J_{\alpha,N}(u_j)\leq 1$. Then \eqref{Jsplit} implies that $\{u_j\}$ is uniformly bounded in $H^{\frac{N}{2}}(\mathbb{S}^N)$. Hence, up to a subsequence, we have
\begin{equation*}
u_j\rightharpoonup u \text{ in }H^{\frac{N}{2}}(\mathbb{S}^N),\qquad  u_j\to u\text{ in }L^q(\mathbb{S}^N)\text{ for any }q>1
\end{equation*}
and $u_j\to u$ almost everywhere for some axially symmetric $u\in H^{\frac{N}{2}}(\mathbb{S}^{N})$.

Choosing $p>1$ and applying Beckner's inequality with $\alpha=1$ to $pu_j$, we have
\begin{equation*}
\log\int_{\mathbb{S}^N}e^{Npu_j}dw\leq \frac{Np^2}{2(N-1)!}\int_{\mathbb{S}^N}u_j P_{r,N}u_jdw.
\end{equation*}
Then $e^{Nu_j}$ is uniformly bounded in $L^p(\mathbb{S}^N)$. Since $u_j\to u$ almost everywhere, by Vitali's convergence theorem, we have
\begin{equation*}
e^{Nu_j}\to e^{Nu}\text{ in }L^1(\mathbb{S}^N).
\end{equation*}

By weak lower semi-continuity of $\int_{\mathbb{S}^N}uP_{r,N}udw$, taking $j\to\infty$, we see that $u$ is a minimizer and $u\in \mathcal{L}_{r,N}$.
\end{proof}

Then we show that $u$ is a smooth function.
\begin{proposition}\label{minimizerreg}
Let $u$ be as given in Proposition \ref{minimizerexist}. Then $u\in C^{\infty}(\mathbb{S}^N)$.
\end{proposition}
\begin{proof}
With the differentiability from Lemma \ref{Differentiable} and the vanishing of the Lagrange multipliers proved by Wei-Xu \cite{WX1998}, we see that $u$ is a weak solution of \eqref{axial}.

For any $p>1$, applying Beckner's inequality with $\alpha=1$ to $pu$ and $-pu$, we have $e^{N|u|}\in L^p(\mathbb{S}^N)$. Since
\begin{equation*}
\int_{\mathbb{S}^N}udw=0,
\end{equation*}
elliptic estimates give
\begin{equation}\label{elliptic-Lp}
\|u\|_{W^{N,p}(\mathbb{S}^N)} \leq C_{N,p}\left(\|P_Nu\|_{L^p(\mathbb{S}^N)}+\|u\|_{L^p(\mathbb{S}^N)}\right).
\end{equation}

For odd $N$, \eqref{elliptic-Lp} is the standard estimate for the classical elliptic pseudodifferential operator obtained from the complex-power construction \cite{Seeley1967}, see also \cite{Taylor1991}. For even $N$ it is the usual elliptic estimate. Taking $p>N$, we have $u\in C^{N-1,\epsilon}$ for some $\epsilon\in (0,1)$. A standard bootstrap argument then gives $u\in C^{\infty}(\mathbb{S}^N)$.

\end{proof}

\subsection{Integral identities}
In this subsection, we establish useful integral identities. Before that, we introduce the following auxiliary functions and quantities that will be used frequently in this paper. Define
\begin{equation*}
G(x):=(1-x^2)u',
\end{equation*}
where $u$ is a solution to \eqref{axial}. 

We expand $G$ in terms of Gegenbauer polynomials
\begin{equation}
\label{Gexpand}
G=a_0F_0+\beta x+a_2F_2(x)+\sum_{k=3}^{\infty}a_kF_k(x),
\end{equation}
and define
\begin{equation*}
g:= \frac{\Gamma(\frac{N+1}{2})}{\sqrt{\pi}\Gamma(\frac{N}{2})}(1-x^2)^{\frac{N-2}{2}} \frac{e^{Nu}}{\gamma}=(1-x^2)^{\frac{N-2}{2}} \frac{e^{Nu}}{\int_{-1}^{-1}(1-x^2)^{\frac{N-2}{2}}e^{Nu}dx}
\end{equation*}
and
\begin{equation*}
a:=\int_{-1}^1 (1-x^2)gdx.    
\end{equation*}

Define the positive operator $\mathcal{Q}_N$ spectrally by 
\begin{equation}\label{QNeigen}
\mathcal{Q}_N F_k=\frac{\Gamma(k+N-1)}{\Gamma(k+1)},\qquad k\geq 0.
\end{equation}
Then we have
\begin{equation*}
P_{r,N}=(-\Delta_{\mathbb{S}^N})Q_N.
\end{equation*}

The identity below will be proved in Appendix \ref{Technicalities}.
\begin{lemma}\label{lemXGPNu}
For any smooth axially symmetric function $u$ and $G=(1-x^2)u'$, we have
\begin{equation}\label{xGPNu}
\int_{\mathbb{S}^N}xGP_{r,N}udw=\frac{N-1}{2}\int_{\mathbb{S}^N}GQ_NGdw.
\end{equation}
\end{lemma}

Then we have the following integral identities.
\begin{lemma}
Let $a$, $g$, and $G$ be as above and $u$ be an axially symmetric solution of \eqref{axial}. Then for every $N\geq 3$, we have $a_0=0$ and the following integral identities
\begin{equation}\label{F1 beta}
    \int_{\mathbb{S}^N}xGdw=\frac{\beta}{N+1},
\end{equation}
\begin{equation}\label{a}
    a=\int_{-1}^{1}(1-x^2)gdx=\frac{N}{N+1}(1-\alpha\beta),
\end{equation}
\begin{equation}\label{bbk}
    \int_{-1}^{1}(1-x^2)^{\frac{N-2}{2}}F_k Gdx=-\frac{2^{N-1}\Gamma(\frac{N}{2})^2 \Gamma(k)}{\alpha \Gamma(k+N)}\int_{-1}^{1}(1-x^2)gF_{k}'dx,\text{ }k\geq 2,
\end{equation}
\begin{equation}\label{D(N-2/2)G}
\int_{\mathbb{S}^N}G\mathcal{Q}_NGdw=\frac{2(N-2)!}{N+1}\left(N+1-\frac{1}{\alpha}\right)\beta,
\end{equation}  
\end{lemma}
\begin{proof}
Test \eqref{axial} against $x$ and use $P_{r,N}x=N!x$, we have
\begin{equation*}
\int_{\mathbb{S^N}}xudw=0.
\end{equation*}

Since $G=\nabla x\cdot \nabla u$ and $\Delta_{\mathbb{S}^N}x=-Nx$, integration by parts yields
\begin{equation*}
\int_{\mathbb{S^N}}Gdw=N\int_{\mathbb{S^N}}xudw=0.
\end{equation*}
Thus $a_0=0$.

Orthogonality and \eqref{Fknorm} give \eqref{F1 beta}.

Since
\begin{equation*}
\text{div}(x\nabla x)=1-(N+1)x^2=-NF_2,
\end{equation*}
we have
\begin{equation}\label{xGF2}
\int_{\mathbb{S^N}}xGdw=N\int_{\mathbb{S^N}}uF_2dw.
\end{equation}

On the other hand, we test \eqref{axial} against $F_2$ and use \eqref{xGF2} and \eqref{F1 beta} to get
\begin{equation*}
\int_{-1}^{1}gF_2dx=\alpha\beta,
\end{equation*}
which is equivalent to \eqref{a}.

Test \eqref{axial} against $F_j$ in the weighted space $L^2((-1,1);(1-x^2)^{\frac{N}{2}-1}dx)$. By the self-adjointness of $P_{r,N}$, we get
\begin{equation*}
\begin{aligned}
\alpha\int_{-1}^{1}(1-x^2)^{\frac{N}{2}-1} u \cdot P_{r,N}F_jdx=
\alpha\int_{-1}^{1}(1-x^2)^{\frac{N}{2}-1} P_{r,N}u \cdot F_jdx=\frac{(N-1)!\sqrt{\pi}\Gamma(\frac{N}{2})}{\Gamma(\frac{N+1}{2})}\int_{-1}^{1}gF_jdx.    
\end{aligned}
\end{equation*}

By \eqref{Paneitzeigen}, we have
\begin{equation*}
\alpha \frac{\Gamma(j+N)}{\Gamma(j)}\int_{-1}^{1}(1-x^2)^{\frac{N}{2}-1} uF_jdx=\frac{(N-1)!\sqrt{\pi}\Gamma(\frac{N}{2})}{\Gamma(\frac{N+1}{2})}\int_{-1}^{1}gF_jdx.
\end{equation*}
Hence, for $j\geq 1$, we have
\begin{equation}\label{uFk}
\int_{-1}^{1}(1-x^2)^{\frac{N}{2}-1} uF_jdx=\frac{2^{N-1}\Gamma(\frac{N}{2})^2\Gamma(j)}{\alpha\Gamma(j+N)}\int_{-1}^{1}gF_jdx.
\end{equation}

Note that by integration by parts, we have
\begin{equation*}
\begin{aligned}
\int_{-1}^{1}(1-x^2)^{\frac{N}{2}-1} GF_kdx=\int_{-1}^{1}(1-x^2)^{\frac{N}{2}} u'F_kdx=\int_{-1}^{1}(1-x^2)^{\frac{N}{2}-1} u(NxF_k-(1-x^2)F_k')dx.    
\end{aligned}
\end{equation*}

Using the recurrence formulas \eqref{recurence1}, \eqref{recurence2} and applying \eqref{uFk} with $j=k-1$ and $j=k+1$, we have
\begin{equation*}
\begin{aligned}
\int_{-1}^{1}(1-x^2)^{\frac{N}{2}-1} GF_kdx=-\frac{2^{N-1}\Gamma(\frac{N}{2})^2\Gamma(k)}{\alpha\Gamma(k+N)}\frac{k(k+N-1)}{2k+N-1}\int_{-1}^{1}g(F_{k-1}-F_{k+1})dx.
\end{aligned}
\end{equation*}
Then \eqref{bbk} follows by the recurrence formula \eqref{recurence2}.

Finally, we prove \eqref{D(N-2/2)G}. Integration of $\text{div}(\frac{e^{Nu}}{\gamma}\nabla x)$ gives
\begin{equation}\label{Gmu}
\int_{\mathbb{S}^N}\frac{e^{Nu}}{\gamma}Gdw=\int_{\mathbb{S}^N}x\frac{e^{Nu}}{\gamma}dw=0.
\end{equation}

Similarly, integration of $\text{div}(x\frac{e^{Nu}}{\gamma}\nabla x)$ gives
\begin{equation}\label{xGmu}
\int_{\mathbb{S}^N}\frac{e^{Nu}}{\gamma}xGdw=\frac{N+1}{N}\int_{\mathbb{S}^N}x^2\frac{e^{Nu}}{\gamma}dw-\frac 1N=\alpha\beta.
\end{equation}

By \eqref{xGPNu}, \eqref{axial}, \eqref{F1 beta}, \eqref{Gmu} and \eqref{xGmu}, we have
\begin{equation}\label{xGQ_NG}
\frac{N-1}{2}\int_{\mathbb{S}^N}GQ_NGdw=\int_{\mathbb{S}^N}xGP_{r,N}udw=\frac{(N-1)!}{\alpha}\left(\alpha-\frac{1}{N+1}\right)\beta
\end{equation}
as desired.
\end{proof}

\section{A weighted $\ell^2$ estimate of Gegenbauer coefficients}\label{weightedl^2est}
In this section, we will establish a weighted $\ell^2$ estimate for the Gegenbauer coefficients 
\begin{equation*}
b_{k}:=a_{k} \sqrt{ \int_{-1}^{1}(1-x^2)^{\frac{N-2}{2}}F_{k}^{2}dx},    
\end{equation*}
where $a_k$ is the $k$-th coefficient in the expansion of $G$, see \eqref{Gexpand}. 

The estimate for $b_k$ is one of the most important ingredients of this paper. In \cite{GHX2021}, Gui-Hu-Xie used \eqref{bbk} and the fact that
\begin{equation*}
|F_k'(x)|\leq |F_k'(1)|=\frac{\lambda_k}{N}
\end{equation*}
to estimate $b_k$ for $k\geq 2$ as follows
\begin{equation*}
\begin{aligned}
    b_{k}^{2}&=\frac{1}{\int_{-1}^{1}(1-x^2)^{\frac{N-2}{2}}F_{k}^{2}}\left[\frac{2^{N-1}\Gamma(\frac{N}{2})^2 \Gamma(k)}{\alpha \Gamma(k+N)}\int_{-1}^{1}(1-x^2)gF_{k}'\right]^2\\
    &\leq \frac{(k+N-2)!(2k+N-1)}{2^{N-1}\Gamma(\frac{N}{2})^2 k!}\left[\frac{2^{N-1}\Gamma(\frac{N}{2})^2 \Gamma(k)}{\alpha \Gamma(k+N)}\frac{\lambda_k}{N}a\right]^2\\
    &=\frac{2^{N-1}\Gamma(\frac{N}{2})^2 (2k+N-1)\Gamma(k+1)}{\alpha^2N^2\Gamma(k+N-1)}a^2.
\end{aligned}    
\end{equation*}
However, their estimates are not strong enough to reach the sharp constant $\alpha=\frac{1}{2}$ for $N=4,6,8$.

The works of Li-Wei-Ye \cite{LWY2022} and Gui-Li-Wei-Ye \cite{GLWY2025} refined this estimate and closed the gap $\alpha=\frac{1}{2}$ for $N=4,6,8$. However, the refined estimates are still difficult to generalize to arbitrary $N$.

We now derive a weighted $\ell^2$ estimate for $b_k$. To this end, we define
\begin{equation*}
A_{k}:=\int_{-1}^{1}(1-x^2)\tilde{F}_{k}'g dx.
\end{equation*}
Recalling the definition of $g$, \eqref{AxialLrN} and \eqref{a}, we have
\begin{equation*}
\int_{-1}^{1}g dx=1,\ \int_{-1}^{1}xg dx=0\text{ and }\int_{-1}^{1}(1-x^2)g dx=a=\frac{N}{N+1}(1-\alpha\beta).
\end{equation*}

By \eqref{bbk}, we have
\begin{equation*}
    \frac{\Gamma(k+N-1)}{\Gamma(k+1)}b_k^2=\frac{2^{N-1}\Gamma(\frac{N}{2})^2}{\alpha^2 N^2}(2k+N-1)A_k^2
\end{equation*}
for $k\geq 2$.

With the above properties of $g$, the next theorem provides a weighted $\ell^2$ estimate on $A_k$ and hence on $b_k$, which yields a lower bound for $a$.
\begin{theorem}\label{bk}
Let $N\geq 3$ be an integer and $\frac{1}{2}\leq \alpha\leq 1$. Let $a$, $A_k$ be as above. Then we have
\begin{equation*}
    Na-(N+1)a^2\leq \sum_{k=2}^{\infty}(2k+N-1)A_k^2<\frac{N^2}{N-1}\frac{a}{2-a}-(N+1)a^2.
\end{equation*}
As a consequence, we have $a> \frac{N-2}{N-1}$. 
\end{theorem}
\begin{proof}
\textbf{Step 1: The lower bound} 

Recall that
\begin{equation*}
    b_k^2=\frac{2^{N-1}\Gamma(\frac{N}{2})^2 \Gamma(k+1)}{(k+N-2)!(2k+N-1)}a_k^2.
\end{equation*}

By \eqref{F1 beta}, \eqref{bbk} and \eqref{D(N-2/2)G}, we have
\begin{equation*}
    \Gamma(N)b_1^2=\frac{2^{N-1}\Gamma(\frac{N}{2})^2}{N+1}\beta^2,
\end{equation*}
\begin{equation*}
    \frac{\Gamma(k+N-1)}{\Gamma(k+1)}b_k^2=\frac{2^{N-1}\Gamma(\frac{N}{2})^2}{\alpha^2 N^2}(2k+N-1)A_k^2
\end{equation*}
for $k\geq 2$ and
\begin{equation*}
    \sum_{k=1}^{\infty}\frac{\Gamma(k+N-1)}{\Gamma(k+1)}b_k^2=\frac{2^{N}\Gamma(\frac{N}{2})^2}{(N+1)(N-1)}(N+1-\frac{1}{\alpha})\beta.
\end{equation*}

Then we have
\begin{equation*}
    \begin{aligned}
        \sum_{k=2}^{\infty}(2k+N-1)A_k^2
        &=\frac{\alpha^2 N^2}{2^{N-1}\Gamma(\frac{N}{2})^2}\sum_{k=2}^{\infty}\frac{\Gamma(k+N-1)}{\Gamma(k+1)}b_k^2\\
        &=\frac{2\alpha^2 N^2}{(N+1)(N-1)}(N+1-\frac{1}{\alpha})\beta-\frac{\alpha^2 N^2}{N+1}\beta^2\\
        &=\frac{2N^2}{(N+1)(N-1)}(\alpha (N+1)-1)\alpha\beta-\frac{N^2}{N+1}\alpha^2\beta^2.
    \end{aligned}
\end{equation*}

Since $a=\frac{N}{N+1}(1-\alpha\beta)$, $\alpha\geq \frac{1}{2}$ and $\beta\geq 0$ by \eqref{D(N-2/2)G}, we have
\begin{equation*}
    \sum_{k=2}^{\infty}(2k+N-1)A_k^2=Na-(N+1)a^2+\frac{2N^2}{N-1}(\alpha-\frac{1}{2})\alpha\beta\geq Na-(N+1)a^2.
\end{equation*}

\textbf{Step 2: The upper bound} 

For the upper bound, set
\begin{equation*}
    B_k=\int_{-1}^{1}x(1-x^2)\tilde F_k'gdx,
\end{equation*}
\begin{equation*}
    S=\sum_{k=1}^{\infty}(2k+N-1)A_k^2,
\end{equation*}
and
\begin{equation*}
    T=\sum_{k=1}^{\infty}(2k+N-1)B_k^2.
\end{equation*}

Since $\int_{-1}^{1}xg dx=0$, we have
\begin{equation*}
    B_k=\int_{-1}^{1}x\left((1-x^2)\tilde F_k'-A_k\right)gdx.
\end{equation*}
Then the Cauchy--Schwarz inequality implies that
\begin{equation*}
    \begin{aligned}
        B_k^2
        \leq& \left(\int_{-1}^{1}x^2 gdx\right)\left(\int_{-1}^{1}\left((1-x^2)\tilde F_k'-A_k\right)^2 gdx\right)\\
        =&(1-a)\int_{-1}^{1}\left((1-x^2)^2(\tilde F_k')^2-A_k^2\right) gdx.
    \end{aligned}
\end{equation*}

By \eqref{KNxx}, we have
\begin{equation*}
    \sum_{k=1}^{\infty}(2k+N-1)\int_{-1}^{1}(1-x^2)^2(\tilde F_k')^2 gdx=\int_{-1}^{1}(1-x^2)^2 K_N(x,x) gdx=\frac{N^2}{N-1}\int_{-1}^{1}(1-x^2) gdx=\frac{N^2}{N-1}a.
\end{equation*}
Hence, we have
\begin{equation*}
    T\leq (1-a)\left(\frac{N^2}{N-1}a-S\right). 
\end{equation*}

On the other hand, by  \eqref{KNxy} and \eqref{absoluteconv}, we can apply the dominated convergence theorem to interchange the order of the kernel series and double integrals to get
\begin{equation*}
    S-T=\int_{-1}^{1}\int_{-1}^{1}(1-x^2)(1-y^2)(1-xy)K_N(x,y)g(x)g(y)dxdy\leq \frac{N^2}{N-1}a^2.
\end{equation*}

Moreover, since $g>0$ on $(-1,1)$ and \eqref{KNxy} is strict when $x\neq y$, the upper bound here is strict.

Eliminating $T$ from the two preceding inequalities, we find that
\begin{equation*}
    S<\frac{N^2}{N-1}\frac{a}{2-a}.
\end{equation*}

Subtracting the $k=1$ term $(N+1)A_1^2=(N+1)a^2$ yields the desired upper bound.

Comparing the lower bound and upper bound yields that $a>\frac{N-2}{N-1}$. This completes the proof.
\end{proof}

\section{Instability of non-constant solutions}\label{Sec: instability}
In this section, we will prove Theorem \ref{instability}. That is, non-constant critical points of $\mathcal{I}_{\alpha,N}$ are unstable. Since for $u\in \mathcal{L}_{r,N}$,
\begin{equation*}
\mathcal{I}_{\alpha,N}(u)=\frac{\sqrt{\pi}\Gamma(\frac{N}{2})}{\Gamma(\frac{N+1}{2})}J_{\alpha,N}(u),
\end{equation*}
we will study the stability of $J_{\alpha,N}$ under axially symmetric perturbations for simplicity.

Before we state the proof, we need to study the second variation of $J_{\alpha,N}$.

Write
\begin{equation*}
d\mu=\frac{e^{Nu}}{\gamma}dw.
\end{equation*}
Then $\mu$ is a probability measure on $\mathbb{S}^N$ and
\begin{equation*}
\int_{\mathbb{S}^N}xd\mu=0.
\end{equation*}

Differentiating $J_{\alpha,N}$ twice, we have for any smooth axially symmetric functions $h_1$ and $h_2$, the second variation of $J_{\alpha,N}$ in axially symmetric class is given by
\begin{equation}\label{secondvariationJ}
\begin{aligned}
\delta^2 J_{\alpha,N}(u)[h_1,h_2]
=&\alpha \int_{\mathbb{S}^N}h_1P_{r,N}h_2dw-N!\left(\int_{\mathbb{S}^N}h_1h_2d\mu-\int_{\mathbb{S}^N}h_1d\mu\int_{\mathbb{S}^N}h_2d\mu\right)  .
\end{aligned}
\end{equation}

Let
\begin{equation*}
C(u):=\int_{\mathbb{S}^N}xe^{Nu}dw
\end{equation*}
be the constraint. 

Suppose $h$ is a smooth axially symmetric function on $\mathbb{S}^N$ with
\begin{equation}\label{htang}
\int_{\mathbb{S}^N}xhd\mu=0.
\end{equation}

Consider $\Psi(s,r):=C(u+sh+rx)$. Since
\begin{equation*}
\Psi_r(0,0)=N\gamma\int_{\mathbb{S}^N}x^2 d\mu>0,
\end{equation*}
the implicit function theorem gives a $C^2$ function $r=r(s)$, defined for small $s$, such that $r(0)=0$,
\begin{equation*}
u_s=u+sh+r(s)x
\end{equation*}
and $C(u_s)=0$.

\begin{lemma}\label{delta2J}
For $u_s$ defined above, we have
\begin{equation*}
\left.\frac{d^2}{ds^2}\right|_{s=0}J_{\alpha,N}(u_s)=\delta^2J_{\alpha,N}(u)[h,h].
\end{equation*}
\end{lemma}
\begin{proof}
Differentiating $C(u_s)$ at $s=0$ and using \eqref{htang}, we have
\begin{equation*}
0=\delta C(u)[h]+r'(0)\delta C(u)[x]=N\int_{\mathbb{S}^N}e^{Nu}dw\int_{\mathbb{S}^N}xhd\mu+r'(0)\delta C(u)[x]=r'(0)\delta C(u)[x].
\end{equation*}

Thus $r'(0)=0$ and $\dot u_0=h$. Here $\dot u_0$ denotes the derivative of $u_s$ in $s$ at $s=0$.

A second differentiation gives
\begin{equation*}
0=\delta^2 C(u)[h,h]+r''(0)\delta C(u)[x].
\end{equation*}

Since $u$ satisfies \eqref{paneitz}, 
\begin{equation*}
\delta J_{\alpha,N}(u)[\varphi]=0    
\end{equation*}
for every axially symmetric variation $\varphi$. Hence,


\begin{equation*}
\begin{aligned}
\left.\frac{d^2}{ds^2}\right|_{s=0}J_{\alpha,N}(u_s)=\delta^2J_{\alpha,N}(u)[h,h]+r''(0)\delta J_{\alpha,N}(u)[x]=\delta^2J_{\alpha,N}(u)[h,h].
\end{aligned}
\end{equation*} 
\end{proof}

Then we give the proof of Theorem \ref{instability}.
\begin{proof}[Proof of Theorem \ref{instability}]

Applying Lemma \ref{PNwf} to \eqref{axial}, since $(1-x^2)(e^{Nu})'=NGe^{Nu}$, we see that $G$ satisfies
\begin{equation*}
\alpha P_{r,N}G=\alpha (1-x^2)(P_{r,N}u)'-NxP_{r,N}u=N!\left(\frac{e^{Nu}}{\gamma}G+x(1-\frac{e^{Nu}}{\gamma})\right).
\end{equation*}

Denote $t:=\alpha\beta$. Testing this equation against $G$ in $L^2(\mathbb{S}^N,dw)$ and using \eqref{F1 beta} and \eqref{xGmu}, we obtain
\begin{equation*}
\frac{\alpha}{N!}\int_{\mathbb{S}^N}GP_{r,N}Gdw=\int_{\mathbb{S}^N}G^2d\mu-t+\frac{t}{\alpha (N+1)}.
\end{equation*}

Let
\begin{equation*}
    \phi:=G-\frac{(N+1)t}{1+Nt}x.
\end{equation*}
Then we have
\begin{equation*}
\int_{\mathbb{S}^N}\phi d\mu=0,\qquad \int_{\mathbb{S}^N}x\phi d\mu=0.
\end{equation*}
Moreover, we have
\begin{equation}\label{phi2}
\int_{\mathbb{S}^N}\phi^2 d\mu=\int_{\mathbb{S}^N}G^2d\mu-\frac{(N+1)t^2}{1+Nt}.
\end{equation}

Since $P_{r,N}x=N!x$ and
\begin{equation*}
\int_{\mathbb{S}^N}x^2 dw=\frac{1}{N+1},\qquad \int_{\mathbb{S}^N}xG dw=\frac{t}{(N+1)\alpha},
\end{equation*}
we find
\begin{equation}\label{phiPrNphi}
\frac{\alpha}{N!}\int_{\mathbb{S}^N}\phi P_{r,N}\phi dw=\int_{\mathbb{S}^N}G^2d\mu-t+\frac{t}{\alpha (N+1)}-\frac{2t^2}{1+Nt}+\frac{\alpha (N+1)t^2}{(1+Nt)^2}.
\end{equation}

Subtracting \eqref{phi2} from \eqref{phiPrNphi}, by \eqref{secondvariationJ}, we have
\begin{equation*}
\frac{1}{N!}\delta^2J_{\alpha,N}(u)[\phi,\phi]=-t+\frac{t}{\alpha (N+1)}-\frac{2t^2}{1+Nt}+\frac{\alpha (N+1)t^2}{(1+Nt)^2}+\frac{(N+1)t^2}{1+Nt}.
\end{equation*}

If $u$ is non-constant, \eqref{D(N-2/2)G} gives $t>0$. Furthermore, by \eqref{a} and Theorem \ref{bk}, we have
\begin{equation*}
\frac{N-2}{N-1}<a=\frac{N}{N+1}(1-t).
\end{equation*}
Then Lemma \ref{tinstability} gives
\begin{equation*}
\delta^2J_{\alpha,N}(u)[\phi,\phi]<0.
\end{equation*}
This implies
\begin{equation*}
0<t<\frac{2}{N(N-1)}.
\end{equation*}

By Lemma \ref{delta2J}, it follows that $u$ is unstable. Consequently, we finish the proof of Theorem \ref{instability}.
\end{proof}

\section{Proof of Theorem \ref{main}}\label{proofmain}
In this section, we will complete the proof of Theorem \ref{main}.

\begin{proof}[Proof of Theorem \ref{main}]
First suppose $\frac12<\alpha<1$.  By Proposition \ref{minimizerexist} and Proposition \ref{minimizerreg}, the infimum is attained by a smooth axially symmetric function $u$.  A minimizer is stable along every curve in the constraint.  Theorem \ref{instability} therefore forces $u$ to be constant, and hence
\begin{equation*}
\inf_{u\in \mathcal{L}_{r,N}}J_{\alpha,N}(u)=0
\end{equation*}
for any $\frac{1}{2}<\alpha<1$. For any fixed $u\in\mathcal{L}_{r,N}$, we have
\begin{equation*}
J_{1/2,N}(u)=\lim_{\alpha\downarrow1/2}J_{\alpha,N}(u)\geq0,
\end{equation*}

If $\alpha\geq1$, Beckner's inequality and positivity of $P_N$ give
\begin{equation}\label{alpha>1}
J_{\alpha,N}(u)
=J_{1,N}(u)
+\frac{\alpha-1}{2}\int_{\mathbb{S}^N}uP_{r,N}u dw
\geq0.
\end{equation}
Constants show that the infimum is zero in every case.

If $\frac12\leq\alpha\leq1$ and
$J_{\alpha,N}(u)=0$, then $u\in \mathcal{L}_{r,N}$ is a global minimizer. Lemma \ref{Differentiable} and the regularity argument in Proposition \ref{minimizerreg} show that $u$ is smooth and stable. Theorem \ref{instability} forces it to be constant.  If $\alpha>1$, equality in \eqref{alpha>1} forces
\begin{equation*}
\int_{\mathbb{S}^N}uP_{r,N}udw=0.
\end{equation*}
Since $P_{r,N}$ is a positive operator, $u$ is constant.  

Conversely, the equality is attained by any constant function.
\end{proof}

\appendix

\section{Optimality of $\alpha=\frac{1}{2}$ for odd $N$}\label{optimal}

In this appendix, we will provide the proof of Proposition \ref{optimalodd}. Though our this appendix, every $O(1)$ term is uniform as $t\to 1^-$.

\begin{proof}[Proof of Proposition \ref{optimalodd}]
We will argue by contradiction. Suppose that $\alpha<\frac{1}{2}$. For $0<t<1$, define
\begin{equation*}
\varphi_t^{\pm}(x):=\log\frac{1-t^2}{1+t^2\mp 2tx}
\end{equation*}
and
\begin{equation*}
T_t(x)=\frac{(1+t^2)x-2t}{1+t^2-2tx}.
\end{equation*}

Then $T_t$ is a diffeomorphism  from $[-1,1]$ to $[-1,1]$ and
\begin{equation*}
(1-x^2)^{\frac{N}{2}-1}e^{N\varphi_t^+}=(1-T_t(x)^2)^{\frac{N}{2}-1}T_t'(x).
\end{equation*}

It follows that by the change of variable $y=T_t(x)$, we have
\begin{equation*}
\int_{-1}^{1}(1-x^2)^{\frac{N}{2}-1}e^{N\varphi_t^+}dx=\int_{-1}^{1}(1-y^2)^{\frac{N-2}{2}}dy=\frac{\sqrt{\pi}\Gamma(\frac{N}{2})}{\Gamma(\frac{N+1}{2})}.
\end{equation*}

Similarly, we have
\begin{equation*}
\int_{-1}^{1}(1-x^2)^{\frac{N}{2}-1}e^{N\varphi_t^-}dx=\frac{\sqrt{\pi}\Gamma(\frac{N}{2})}{\Gamma(\frac{N+1}{2})}.
\end{equation*}

\textbf{Claim 1: $\varphi_t^\pm$ satisfies \eqref{axial} with $\alpha=1$.} 
Indeed, we have
\begin{equation*}
\frac{d^k}{dx^k}\varphi_t^+(x)=(k-1)!2^kt^k(1+t^2-2tx)^{-k}
\end{equation*}
and
\begin{equation*}
\frac{d^k}{dx^k}e^{N\varphi_t^+(x)}=(N)_k2^kt^k(1-t^2)^{N}(1+t^2-2tx)^{-N-k},
\end{equation*}
where 
\begin{equation*}
(s)_k:=
\begin{cases}
    s(s+1)\cdots (s+k-1),&\qquad k\geq 1,\\
    1,&\qquad k=0
\end{cases}    
\end{equation*}
is the Pochhammer symbol for $s\in \mathbb{R}$ and $k\in \mathbb{N}$.

Using \eqref{Rodrigues} and integrating by parts for $k$ times, we get for $k\geq 1$,
\begin{equation*}
\int_{-1}^{1}(1-x^2)^{\frac{N}{2}-1}\varphi_t^+F_k dx=\kappa_{k,N}(k-1)!2^kt^k\int_{-1}^{+1}(1-x^2)^{\frac{N}{2}+k-1}(1+t^2-2tx)^{-k}dx
\end{equation*}
and
\begin{equation*}
\int_{-1}^{1}(1-x^2)^{\frac{N}{2}-1}e^{N\varphi_t^+}F_k dx=\kappa_{k,N}(N)_k2^kt^k(1-t^2)^{N}\int_{-1}^{+1}(1-x^2)^{\frac{N}{2}+k-1}(1+t^2-2tx)^{-N-k}dx
\end{equation*}
for some constant $\kappa_{k,N}>0$.

With the change of variable $y=T_t(x)$, we have
\begin{equation*}
\begin{aligned}
(1-t^2)^{N}\int_{-1}^{+1}(1-x^2)^{\frac{N}{2}+k-1}(1+t^2-2tx)^{-N-k}dx
=&\int_{-1}^{+1}(1-y^2)^{\frac{N}{2}+k-1}(1+t^2+2ty)^{-k}dy\\
=&\int_{-1}^{+1}(1-y^2)^{\frac{N}{2}+k-1}(1+t^2-2ty)^{-k}dy.
\end{aligned}
\end{equation*}

Then we have
\begin{equation*}
\frac{\Gamma(k+N)}{\Gamma(k)}\int_{-1}^{1}(1-x^2)^{\frac{N}{2}-1}\varphi_t^+F_k dx=(N-1)!\int_{-1}^{1}(1-x^2)^{\frac{N}{2}-1}e^{N\varphi_t^+}F_k dx.
\end{equation*}

By \eqref{Paneitzeigen}, it follows that
\begin{equation*}
\int_{-1}^{1}(1-x^2)^{\frac{N}{2}-1}P_{r,N}\varphi_t^+F_k dx=\int_{-1}^{1}(1-x^2)^{\frac{N}{2}-1}\varphi_t^+P_{r,N}F_k dx=(N-1)!\int_{-1}^{1}(1-x^2)^{\frac{N}{2}-1}e^{N\varphi_t^+}F_k dx.
\end{equation*}

Since
\begin{equation*}
\int_{-1}^{1}(1-x^2)^{\frac{N}{2}-1}F_k=0,
\end{equation*}
we have
\begin{equation*}
\int_{-1}^{1}(1-x^2)^{\frac{N}{2}-1}P_{r,N}\varphi_t^+F_k dx=(N-1)!\int_{-1}^{1}(1-x^2)^{\frac{N}{2}-1}(e^{N\varphi_t^+}-1)F_k dx
\end{equation*}
for $k\geq 1$. For $k=0$, both sides are zero and the equality holds automatically. The same equality also holds for $\varphi_t^-$. Hence the claim follows.

With the change of variable $y=T_t(x)$, we can compute
\begin{equation*}
\int_{-1}^{1}(1-x^2)^{\frac{N}{2}-1}\varphi_t^{\pm}dx=\frac{\sqrt{\pi}\Gamma(\frac{N}{2})}{\Gamma(\frac{N+1}{2})}\log(1-t^2)+O(1)
\end{equation*}
and
\begin{equation*}
\int_{-1}^{1}(1-x^2)^{\frac{N}{2}-1}\varphi_t^{\pm}e^{N\varphi_t^\pm}dx=-\frac{\sqrt{\pi}\Gamma(\frac{N}{2})}{\Gamma(\frac{N+1}{2})}\log(1-t^2)+O(1).
\end{equation*}

\textbf{Claim 2: }
\begin{equation}
\int_{-1}^{1}(1-x^2)^{\frac{N}{2}-1}\varphi_t^{\pm}P_{r,N}\varphi_t^{\pm}dx=-2(N-1)!\frac{\sqrt{\pi}\Gamma(\frac{N}{2})}{\Gamma(\frac{N+1}{2})}\log(1-t^2)+O(1)
\end{equation}
Indeed, write $\varphi_t^+=\log(1-t^2)+v_t^+$ with $v_t=-\log(1+t^2-2tx)$. Since $\varphi_t^+$ satisfies \eqref{axial} with $\alpha=1$, testing it against $\varphi_t^+$, we obtain that
\begin{equation*}
\begin{aligned}
\int_{-1}^{1}(1-x^2)^{\frac{N}{2}-1}\varphi_t^{+}P_{r,N}\varphi_t^{+}dx
=&(N-1)!\int_{-1}^{1}(1-x^2)^{\frac{N}{2}-1}\varphi_t^{+}(e^{N\varphi_t^+}-1)dx\\
=&(N-1)!\int_{-1}^{1}(1-x^2)^{\frac{N}{2}-1}v_t^{+}(e^{N\varphi_t^+}-1)dx
\end{aligned}
\end{equation*}

Since $|v_t^+(x)|\leq C(1+|\log(1-x)|)$, we have
\begin{equation}\label{vt}
\int_{-1}^{1}(1-x^2)^{\frac{N}{2}-1}v_t^{+}dx=O(1).    
\end{equation}

Let $x=\frac{(1+t)^2-(1-t)^2r}{(1+t)^2+(1-t)^2r}$, we have
\begin{equation}\label{vtvarphit}
\begin{aligned}
\int_{-1}^{1}(1-x^2)^{\frac{N}{2}-1}v_t^{+}e^{N\varphi_t^+}dx
=&2^{N-1}\int_0^{\infty}\frac{r^{\frac{N}{2}-1}}{(1+r)^N}\left[-2\log(1-t)-\log(1+r)+\log(1+\frac{(1-t)^2}{(1+t)^2}r)\right]dr\\
=&-\frac{2\sqrt{\pi}\Gamma(\frac{N}{2})}{\Gamma(\frac{N+1}{2})}\log(1-t^2)+O(1).
\end{aligned}
\end{equation}

Combining \eqref{vt} and \eqref{vtvarphit}, Claim 2 follows for $\varphi_t^+$. The proof for $\varphi_t^-$ is similar.

\textbf{Claim 3: }
\begin{equation*}
\int_{-1}^{1}(1-x^2)^{\frac{N}{2}-1}\varphi_t^{+}P_{r,N}\varphi_t^{-}dx=O(1).
\end{equation*}
Indeed, testing \eqref{axial} with $\alpha=1$ for $\varphi_t^-$ against $\varphi_t^+$, we have
\begin{equation*}
\begin{aligned}
\int_{-1}^{1}(1-x^2)^{\frac{N}{2}-1}\varphi_t^{+}P_{r,N}\varphi_t^{-}dx
=&(N-1)!\int_{-1}^{1}(1-x^2)^{\frac{N}{2}-1}\varphi_t^{+}(e^{N\varphi_t^-}-1)dx\\
=&(N-1)!\int_{-1}^{1}(1-x^2)^{\frac{N}{2}-1}v_t^{+}(e^{N\varphi_t^-}-1)dx.
\end{aligned}
\end{equation*}

For $x\leq 0$, since $v_t^+$ is uniformly bounded and
\begin{equation}\label{varphit}
\int_{-1}^{1} (1-x^2)^{\frac{N}{2}-1}e^{N\varphi_t^-}dx=\frac{\sqrt{\pi}\Gamma(\frac{N}{2})}{\Gamma(\frac{N+1}{2})},  
\end{equation}
we have
\begin{equation*}
\int_{-1}^{0} (1-x^2)^{\frac{N}{2}-1}|v_t^{+}|e^{N\varphi_t^-}dx\leq C.  
\end{equation*}

For $x\geq 0$, we have
\begin{equation*}
e^{\varphi_t^-}\leq 1-t^2.
\end{equation*}
Hence, 
\begin{equation*}
\int_{0}^{1} (1-x^2)^{\frac{N}{2}-1}|v_t^{+}|e^{N\varphi_t^-}dx\leq C(1-t^2)^N\int_{0}^{1} (1-x^2)^{\frac{N}{2}-1}|v_t^{+}|dx  \leq C(1-t^2)^N.
\end{equation*}

Combining with \eqref{vt}, Claim 3 follows.

Finally, define 
\begin{equation*}
u_t(x)=\varphi_t^++\varphi_t^--2\log(1-t^2).
\end{equation*}
Then $u_t$ is an even function with
\begin{equation*}
\int_{-1}^{1}x(1-x^2)^{\frac{N}{2}-1}e^{Nu_t}dx=0.
\end{equation*}
Hence, $u_t\in \mathcal{L}_{r,N}$.

\textbf{Claim 4: }
\begin{equation*}
\log\left(\frac{\Gamma(\frac{N+1}{2})}{\sqrt{\pi}\Gamma(\frac{N}{2})}\int_{-1}^{1}(1-x^2)^\frac{N-2}{2} e^{Nu_t} dx\right)=-N\log(1-t^2)+O(1).
\end{equation*}

Indeed, by definition of $u_t$, we have
\begin{equation*}
4^{-N}(1-t^2)^{-N}\max\{e^{N\varphi_t^+},e^{N\varphi_t^-}\}\leq e^{Nu_t}\leq (1-t^2)^{-N}\max\{e^{N\varphi_t^+},e^{N\varphi_t^-}\}.
\end{equation*}
Hence, we have
\begin{equation*}
4^{-N-1}(1-t^2)^{-N}\left(e^{N\varphi_t^+}+e^{N\varphi_t^-}\right)\leq e^{Nu_t}\leq (1-t^2)^{-N}\left(e^{N\varphi_t^+}+e^{N\varphi_t^-}\right).
\end{equation*}

By \eqref{varphit}, we have
\begin{equation*}
c(1-t^2)^{-N}\leq \frac{\Gamma(\frac{N+1}{2})}{\sqrt{\pi}\Gamma(\frac{N}{2})}\int_{-1}^{1}(1-x^2)^\frac{N-2}{2} e^{Nu_t} dx\leq C(1-t^2)^{-N}
\end{equation*}
for some $c,C>0$. Then Claim 4 follows.

By Claim 2 and Claim 3, we have
\begin{equation*}
\begin{aligned}
\int_{-1}^{1}(1-x^2)^{\frac{N}{2}-1}u_tP_{r,N}u_tdx
=&\int_{-1}^{1}(1-x^2)^{\frac{N}{2}-1}\varphi_t^+P_{r,N}\varphi_t^+dx+\int_{-1}^{1}(1-x^2)^{\frac{N}{2}-1}\varphi_t^-P_{r,N}\varphi_t^-dx\\
+&2\int_{-1}^{1}(1-x^2)^{\frac{N}{2}-1}\varphi_t^+P_{r,N}\varphi_t^-dx\\
=&-4(N-1)!\frac{\sqrt{\pi}\Gamma(\frac{N}{2})}{\Gamma(\frac{N+1}{2})}\log(1-t^2)+O(1).
\end{aligned}
\end{equation*}

Meanwhile, since
\begin{equation*}
u_t=-\log(1+t^2-2tx)-\log(1+t^2+2tx),
\end{equation*}
we have
\begin{equation*}
\int_{-1}^{1}(1-x^2)^{\frac{N}{2}-1}u_tdx=O(1).
\end{equation*}

Combining with Claim 4, we obtain
\begin{equation*}
\mathcal{I}_{\alpha,N}(u_t)=-(N-1)!\frac{\sqrt{\pi}\Gamma(\frac{N}{2})}{\Gamma(\frac{N+1}{2})}(2\alpha-1)\log(1-t^2)+O(1).
\end{equation*}

If $\alpha<\frac{1}{2}$, then we have $\mathcal{I}_{\alpha,N}(u_t)\to -\infty$ as $t\to 1^-$, a contradiction.
\end{proof}

\section{Technicalities}\label{Technicalities}
In this appendix, we will include several technicalities. 
\begin{lemma}\label{PNwf}
For any integer $N\geq 3$ and smooth function $f$ on $[-1,1]$, we have
\begin{equation*}
P_{r,N}((1-x^2)f')=(1-x^2)(P_{r,N}f)'-NxP_{r,N}f.
\end{equation*}
\end{lemma}

\begin{proof}
It suffices to verify it for $f=F_k=F_k^{\frac{N-1}{2}}$.

For $k=0$, we have $F_0=1$. By \eqref{Paneitzeigen}, both sides vanish.

For $k\geq 1$, by \eqref{Paneitzeigen} and \eqref{recurence2}, we have
\begin{equation*}
P_{r,N}((1-x^2)F_k')=\frac{\lambda_k}{2k+N-1}\left(\frac{\Gamma(k+N-1)}{\Gamma(k-1)}F_{k-1}-\frac{\Gamma(k+N+1)}{\Gamma(k+1)}F_{k+1}\right).
\end{equation*}

On the other hand, by \eqref{Paneitzeigen}, \eqref{recurence1} and \eqref{recurence2}, we get
\begin{equation*}
\begin{aligned}
(1-x^2)(P_{r,N}F_k)'-NxP_{r,N}F_k
&=\frac{\Gamma(k+N)}{\Gamma(k)}((1-x^2)F_k'-NxF_k)\\
&=\frac{\Gamma(k+N)}{\Gamma(k)(2k+N-1)}(k(k-1)F_{k-1}-(k+N-1)(k+N)F_{k+1})\\
&=\frac{\lambda_k}{2k+N-1}\left(\frac{\Gamma(k+N-1)}{\Gamma(k-1)}F_{k-1}-\frac{\Gamma(k+N+1)}{\Gamma(k+1)}F_{k+1}\right).
\end{aligned}    
\end{equation*}
The desired result follows.
\end{proof}

Next, we prove Lemma \ref{lemXGPNu}.
\begin{proof}[Proof of Lemma \ref{lemXGPNu}].\
Let
\begin{equation*}
r_k:=\frac{k(k+N-1)}{2k+N-1},
\end{equation*}
and $r_0:=0$.

Write
\begin{equation*}
u=\sum_{k=0}^{\infty}d_kF_k,
\qquad
G=(1-x^2)u'=\sum_{j=0}^{\infty}a_jF_j,
\qquad
xG=\sum_{k=0}^{\infty}e_kF_k.
\end{equation*}
Then \eqref{recurence1} and \eqref{recurence2} give
\begin{equation}\label{aj-dj}
a_j=r_{j+1}d_{j+1}-r_{j-1}d_{j-1}
\end{equation}
and
\begin{equation}\label{ek-ak}
e_k=\frac{k+1}{2k+N+1}a_{k+1}
+\frac{k+N-2}{2k+N-3}a_{k-1}
\end{equation}
for $k\geq 1$, where terms with negative indices are omitted.  From \eqref{Fknorm}, \eqref{Paneitzeigen} and \eqref{QNeigen}, we have
\begin{equation}\label{normsimplification1}
\frac{\Gamma(k+N-1)}{\Gamma(k+1)}\int_{\mathbb{S}^N}F_k^2dw=\frac{(N-1)!}{2k+N-1},
\end{equation}
and
\begin{equation}\label{normsimplification2}
\frac{\Gamma(k+N)}{\Gamma(k)}\int_{\mathbb{S}^N}F_k^2dw=(N-1)!r_k.
\end{equation}

Set
\begin{equation*}
L=\int_{\mathbb{S}^N}xGP_{r,N}u dw
=\sum_{k=1}^{\infty}\frac{\Gamma(k+N)}{\Gamma(k)}d_ke_k\int_{\mathbb{S}^N}F_k^2dw.
\end{equation*}
Substitution of \eqref{aj-dj}--\eqref{normsimplification2} shows that the coefficient of $d_k^2$ in $L$ is
\begin{equation*}
\mathcal D_k
=\frac{(N-1)(N-1)!}{2}r_k^2
\left(\frac1{2k+N-3}+\frac1{2k+N+1}\right),
\end{equation*}
while the coefficient of $d_kd_{k+2}$ is
\begin{equation*}
\mathcal C_k
=-\frac{(N-1)(N-1)!}{2k+N+1}r_kr_{k+2}.
\end{equation*}
There are no other mixed terms.

On the other hand,
\begin{equation*}
\begin{aligned}
R
=&\frac{N-1}{2}\int_{\mathbb{S}^N}GQ_NGdw\\
=&\frac{N-1}{2}\sum_{j=0}^{\infty}\frac{\Gamma(j+N-1)}{\Gamma(j+1)}\int_{\mathbb{S}^N}F_j^2dw
(r_{j+1}d_{j+1}-r_{j-1}d_{j-1})^2.    
\end{aligned}
\end{equation*}
Using \eqref{normsimplification1}, its $d_k^2$ coefficient is
\begin{equation*}
\frac{N-1}{2}r_k^2\left(\frac{\Gamma(k+N-2)}{\Gamma(k)}\int_{\mathbb{S}^N}F_{k-1}^2dw+\frac{\Gamma(k+N)}{\Gamma(k+2)}\int_{\mathbb{S}^N}F_{k+1}^2dw\right)=\mathcal D_k,
\end{equation*}
and its $d_kd_{k+2}$ coefficient is
\begin{equation*}
-(N-1)\frac{\Gamma(k+N)}{\Gamma(k+2)}r_kr_{k+2}\int_{\mathbb{S}^N}F_{k+1}^2dw =\mathcal C_k.
\end{equation*}
Thus $L=R$.    
\end{proof}

\begin{lemma}\label{Differentiable}
The functionals
\begin{equation*}
Z(u):=\int_{\mathbb{S}^N}e^{Nu}dw,\qquad C(u):=\int_{\mathbb{S}^N}xe^{Nu}dw
\end{equation*}
are continuously Fr\'echet differentiable on the axially symmetric subspace of $H^{\frac{N}{2}}(\mathbb{S}^N)$ with
\begin{equation*}
\delta Z(u)[h]=N\int_{\mathbb{S}^N}he^{Nu}dw,\qquad \delta C(u)[h]:=N\int_{\mathbb{S}^N}xhe^{Nu}dw.
\end{equation*}
\end{lemma}
\begin{proof}
Note that for any real number $p,q$, we have
\begin{equation*}
|e^{p+q}-e^p-qe^p|\leq\frac12 q^2e^{|p|+|q|}.
\end{equation*}

With $p=Nu$ and $q=Nh$, H\"older's inequality gives
\begin{align*}
\int_{\mathbb{S}^N}|h|^2e^{N|u|+N|h|}dw
&\leq \|h\|_{L^6(\mathbb{S}^N)}^2
\left(\int_{\mathbb{S}^N}e^{3N|u|}dw\right)^{1/3}
\left(\int_{\mathbb{S}^N}e^{3N|h|}dw\right)^{1/3}.
\end{align*}

By Sobolev embedding $H^{N/2}(\mathbb{S}^N)\hookrightarrow L^6(\mathbb{S}^N)$ and Beckner's inequality with $\alpha=1$ applied to $\pm3u$ and $\pm3h$, we have the differentiability of $Z(u)$. To prove continuity of $\delta Z(u)$, note that if $u_j\to u$ in $H^{N/2}$, the inequality
\begin{equation*}
|e^{Nu_j}-e^{Nu}|
\leq N|u_j-u|e^{N|u|+N|u_j-u|}
\end{equation*}
along with the same H\"older's inequality and locally uniform exponential bounds imply 
\begin{equation*}
e^{Nu_j}\to e^{Nu} \text{ in } L^{\frac{6}{5}}(\mathbb{S}^N).    
\end{equation*}

A similar argument also holds for $\delta C(u)$.

\end{proof}

\begin{lemma}\label{tinstability}
Let $N\geq 3$, $\frac{1}{2}\leq \alpha\leq 1$ and $0<t<\frac{2}{N(N-1)}$. Then
\begin{equation*}
-t+\frac{t}{\alpha(N+1)} -\frac{2t^2}{1+Nt} +\frac{\alpha(N+1)t^2}{(1+Nt)^2} +\frac{(N+1)t^2}{1+Nt}<0
\end{equation*}
\end{lemma}
\begin{proof}
Denote the left-hand side by $H_N(\alpha,t)$. For fixed $t>0$, we have
\begin{equation*}
\frac{\partial^2}{\partial\alpha^2}H_N(\alpha,t)=\frac{2t}{(N+1)\alpha^3}>0.
\end{equation*}
Hence, $H_N$ is strictly convex on $[\frac{1}{2},1]$. It suffices to check $H_N<0$ at $\alpha=1$ and $\alpha=\frac{1}{2}$.

At $\alpha=1$, we have
\begin{equation*}
H_N(1,t)=-\frac{Nt(1-t)^2}{(N+1)(1+Nt)^2}<0
\end{equation*}
since $0<t<\frac{2}{N(N-1)}\leq \frac{1}{3}$.

At $\alpha=\frac{1}{2}$, we have
\begin{equation*}
H_N(\frac{1}{2},t)=-\frac{t}{2(N+1)(1+Nt)^2}[-2N(N-1)t^2+(N^2-6N+1)t+2(N-1)]<0.
\end{equation*}
This completes the proof.
\end{proof}

\section*{Acknowledgements}
 The research of  C. Gui is supported by NSFC Key Program (Grant No.12531010), University of Macau research grants CPG2024-00016-FST, CPG2025-00032-FST, CPG2026-00027-FST, SRG2023-00011-FST, MYRGGRG2023-00139-FST-UMDF, UMDF Professorial Fellowship
of Mathematics, Macao SAR FDCT 0003/2023/RIA1 and Macao SAR FDCT 0024/2023/RIB1. The research of J. Wei is partially supported by General Research Grant of HKSAR GRF 14309824. The authors acknowledge the use of AI tools. All mathematical arguments and proofs in the final manuscript were checked and written by the authors.


\medskip

\begin{thebibliography}{99}


\bibitem{Beckner1993}
W.~Beckner,
\emph{Sharp Sobolev inequalities on the sphere and the Moser--Trudinger inequality},
Ann. of Math. (2) \textbf{138} (1993), no.~1, 213--242.

\bibitem{CLY2019}
J.~S.~Case, Y.-J.~Lin, and W.~Yuan,
\emph{Conformally variational Riemannian invariants},
Trans. Amer. Math. Soc. \textbf{371} (2019), no.~11, 8217--8254.

\bibitem{CM2023}
J.~S.~Case and A.~Malchiodi,
\emph{A factorization of the GJMS operators of special Einstein products and applications},
J. Lond. Math. Soc. (2) \textbf{110} (2024), no.~5,
Paper No.~e70023, 17 pp.

\bibitem{ChangGui202}
S.-Y.~A.~Chang and C.~Gui,
\emph{A sharp inequality on the exponentiation of functions on the sphere},
Comm. Pure Appl. Math. \textbf{76} (2023), no.~6, 1303--1326.

\bibitem{ChangHang2022}
S.-Y.~A.~Chang and F.~B.~Hang,
\emph{Improved Moser--Trudinger--Onofri inequality under constraints},
Comm. Pure Appl. Math. \textbf{75} (2022), no.~1, 197--220.

\bibitem{ChangYang1987}
S.-Y.~A.~Chang and P.~C.~Yang,
\emph{Prescribing Gaussian curvature on $S^2$},
Acta Math. \textbf{159} (1987), nos.~3--4, 215--259.

\bibitem{ChangYang1988}
S.-Y.~A.~Chang and P.~C.~Yang,
\emph{Conformal deformation of metrics on $S^2$},
J. Differential Geom. \textbf{27} (1988), no.~2, 259--296.

\bibitem{ChangYang1995}
S.-Y.~A.~Chang and P.~C.~Yang,
\emph{Extremal metrics of zeta function determinants on $4$-manifolds},
Ann. of Math. (2) \textbf{142} (1995), no.~1, 171--212.

\bibitem{ChangYang1997}
S.-Y.~A.~Chang and P.~C.~Yang,
\emph{On uniqueness of solutions of $n$th-order differential equations in conformal geometry},
Math. Res. Lett. \textbf{4} (1997), no.~1, 91--102.

\bibitem{DHL2000}
Z.~Djadli, E.~Hebey, and M.~Ledoux,
\emph{Paneitz-type operators and applications},
Duke Math. J. \textbf{104} (2000), no.~1, 129--169.

\bibitem{DM2008}
Z.~Djadli and A.~Malchiodi,
\emph{Existence of conformal metrics with constant $Q$-curvature},
Ann. of Math. (2) \textbf{168} (2008), no.~3, 813--858.


\bibitem{FFGG1998}
J.~Feldman, R.~Froese, N.~Ghoussoub, and C.~Gui,
\emph{An improved Moser--Aubin--Onofri inequality for axially symmetric functions on $S^2$},
Calc. Var. Partial Differential Equations \textbf{6} (1998), no.~2, 95--104.

\bibitem{FG2013}
C.~Fefferman and C.~R.~Graham,
\emph{Juhl's formulae for GJMS operators and $Q$-curvatures},
J. Amer. Math. Soc. \textbf{26} (2013), no.~4, 1191--1207.

\bibitem{FrankLieb2012}
R.~L. Frank and E.~H. Lieb,
\emph{A new, rearrangement-free proof of the sharp Hardy--Littlewood--Sobolev inequality},
in \emph{Spectral Theory, Function Spaces and Inequalities},
Oper. Theory Adv. Appl., Vol.~219,
Birkh\"auser/Springer Basel AG, Basel, 2012, pp.~55--67.

\bibitem{GL2010}
N.~Ghoussoub and C.-S.~Lin,
\emph{On the best constant in the Moser--Onofri--Aubin inequality},
Comm. Math. Phys. \textbf{298} (2010), no.~3, 869--878.

\bibitem{GJMS1992}
C.~R.~Graham, R.~Jenne, L.~J.~Mason, and G.~A.~J.~Sparling,
\emph{Conformally invariant powers of the Laplacian. I. Existence},
J. Lond. Math. Soc. (2) \textbf{46} (1992), no.~3, 557--565.

\bibitem{GM2018}
C.~Gui and A.~Moradifam,
\emph{The sphere covering inequality and its applications},
Invent. Math. \textbf{214} (2018), no.~3, 1169--1204.

\bibitem{GHM2020}
C.~Gui, F.~Hang, and A.~Moradifam,
\emph{The sphere covering inequality and its dual},
Comm. Pure Appl. Math. \textbf{73} (2020), no.~12, 2685--2707.

\bibitem{GHX2021}
C.~Gui, Y.~Hu, and W.~Xie,
\emph{Improved Beckner's inequality for axially symmetric functions on $\mathbb{S}^4$},
Rev. Mat. Iberoam. \textbf{40} (2024), no.~1, 355--388.

\bibitem{GHW2022}
C.~Gui, Y.~Hu, and W.~Xie,
\emph{Improved Beckner's inequality for axially symmetric functions on $\mathbb{S}^n$},
J. Funct. Anal. \textbf{282} (2022), no.~5,
Paper No.~109335, 47 pp.

\bibitem{GLWY2025}
C.~Gui, T.~Li, J.~Wei, and Z.~Ye,
\emph{Sharp Beckner's inequalities for axially symmetric functions on $\mathbb{S}^6$ and $\mathbb{S}^8$},
Adv. Math. \textbf{480} (2025), Paper No.~110487, 64 pp.


\bibitem{GW2000}
C.~Gui and J.~Wei,
\emph{On a sharp Moser--Aubin--Onofri inequality for functions on $S^2$ with symmetry},
Pacific J. Math. \textbf{194} (2000), no.~2, 349--358.

\bibitem{GurMal2015}
M.~J.~Gursky and A.~Malchiodi,
\emph{A strong maximum principle for the Paneitz operator and a non-local flow for the $Q$-curvature},
J. Eur. Math. Soc. (JEMS) \textbf{17} (2015), no.~9, 2137--2173.

\bibitem{JinLiXiong2017}
T.~Jin, Y.~Y.~Li, and J.~Xiong,
\emph{The Nirenberg problem and its generalizations: a unified approach},
Math. Ann. \textbf{369} (2017), nos.~1--2, 109--151.

\bibitem{LWY2022}
T.~Li, J.~Wei, and Z.~Ye,
\emph{On sharp Beckner's inequality for axially symmetric functions on $\mathbb{S}^4$},
Math. Res. Lett. \textbf{31} (2024), no.~5, 1523--1550.

\bibitem{LiXiong2019}
Y.~Y.~Li and J.~Xiong,
\emph{Compactness of conformal metrics with constant $Q$-curvature. I},
Adv. Math. \textbf{345} (2019), 116--160.

\bibitem{Mal2006}
A.~Malchiodi,
\emph{Compactness of solutions to some geometric fourth-order equations},
J. Reine Angew. Math. \textbf{594} (2006), 137--174.

\bibitem{Mori1998}
M.~Morimoto,
\emph{Analytic Functionals on the Sphere},
Translations of Mathematical Monographs, Vol.~178,
American Mathematical Society, Providence, RI, 1998.


\bibitem{Olver2010}
F.~W.~J.~Olver, D.~W.~Lozier, R.~F.~Boisvert, and C.~W.~Clark, eds.,
\emph{NIST Handbook of Mathematical Functions},
Cambridge University Press, New York, NY, 2010.

\bibitem{Paneitz2008}
S.~M.~Paneitz,
\emph{A Quartic Conformally Covariant Differential Operator for Arbitrary
    Pseudo-Riemannian Manifolds (Summary)},
SIGMA Symmetry Integrability Geom. Methods Appl. \textbf{4} (2008),
Paper No.~036, 3 pp.

\bibitem{Seeley1967}
R.~T. Seeley,
\emph{Complex powers of an elliptic operator},
in \emph{Singular Integrals (Proc. Sympos. Pure Math., Chicago, Ill., 1966)},
Proc. Sympos. Pure Math., Vol.~10, Amer. Math. Soc., Providence, RI, 1967, pp.~288--307.

\bibitem{SSTW2019}
Y.~Shi, J.~Sun, G.~Tian, and D.~Wei,
\emph{Uniqueness of the mean field equation and rigidity of Hawking mass},
Calc. Var. Partial Differential Equations \textbf{58} (2019), no.~2,
Paper No.~41, 16 pp.


\bibitem{Taylor1991}
M.~E. Taylor,
\emph{Pseudodifferential Operators and Nonlinear PDE},
Progress in Mathematics, Vol.~100, Birkh\"auser Boston, Boston, MA, 1991.

\bibitem{Widom1988}
H.~Widom,
\emph{On an inequality of Osgood, Phillips and Sarnak},
Proc. Amer. Math. Soc. \textbf{102} (1988), no.~3, 773--774.

\bibitem{WX1998}
J.~Wei and X.~Xu,
\emph{On conformal deformations of metrics on $S^n$},
J. Funct. Anal. \textbf{157} (1998), no.~1, 292--325.

\bibitem{Zhang2025}
S. Zhang, \emph{The moving plane method and the uniqueness of high-order elliptic equation with GJMS operator}, Proc. Roy. Soc. Edinburgh Sect. A, First View (2025), 1–45,

\end{thebibliography}
\end{document}